\documentclass[12pt]{article}

\usepackage{amsmath}
\usepackage{amssymb}
\usepackage[mathscr]{eucal}
\usepackage{graphicx}
\usepackage{color}

\usepackage{bm}
\usepackage{mathrsfs}
\usepackage{amsthm}
\usepackage{enumerate}
\usepackage[mathscr]{eucal}
\usepackage{eqlist}
\usepackage{graphicx}

\allowdisplaybreaks
\numberwithin{equation}{section}

\newtheorem{Theorem}{\sc Theorem}
\newtheorem{Definition}[Theorem]{\sc Definition}
\newtheorem{Proposition}[Theorem]{\sc Proposition}

\newcommand{\R}{{\if mm {\rm I}\mkern -3mu{\rm R}\else \leavevmode
        \hbox{I}\kern -.17em\hbox{R} \fi}}

\newcommand{\wto}{ \ \stackrel{w} {\longrightarrow} \ }
\def\sqr#1#2{{
        \vcenter{
            \vbox{\hrule height.#2pt
                \hbox{\vrule width.#2pt height#1pt \kern#1pt
                    \vrule width.#2pt
                }
                \hrule height.#2pt
            }
        }
}}

\def\esssup{\mathop{\rm ess\, sup}}

\def\real{\mathbb{R}}

\def\lista#1
{{ \itemindent 0.0cm \labelsep .2cm \leftmargin 0.8cm \rightmargin
        0.0cm \labelwidth 0.6cm \topsep 0.0mm
        \parsep 0.0mm
        \itemsep 0.0mm
        \begin{list}{}
            { \setlength{\leftmargin}{.8cm} \setlength{\rightmargin}{0.0cm}
                \setlength{\parsep}{0.0mm} \setlength{\topsep}{.0mm}
                \setlength{\parskip}{.0cm} \setlength{\itemsep}{.0cm} }
            {#1}\end{list}} }

\begin{document}
\title{ Differential inclusion systems with double phase competing operators, convection, and mixed boundary conditions}
\author{
Jinxia Cen\footnote{School of Mathematical Sciences, Zhejiang Normal University, Jinhua 321004, P.R. China. E-mail: jinxcen@163.com},\; 
Salvatore A. Marano\footnote{Dipartimento di Matematica e Informatica, Università degli Studi di Catania, Viale A. Doria 6, 95125 Catania, Italy. E-mail: marano@dmi.unict.it},\; Shengda Zeng\footnote{Corresponding author. National Center for Applied Mathematics in Chongqing and School of Mathematical Sciences, Chongqing Normal University, Chongqing 401331, P.R. China. E-mail:\\ zengshengda@163.com}
}
\date{}
\maketitle
\begin{abstract}
In this paper, a new framework for studying the existence of generalized or strongly generalized solutions to a wide class of inclusion systems involving double-phase, possibly competing differential operators, convection, and mixed boundary conditions is introduced. The technical approach exploits Galerkin's method and a surjective theorem for multifunctions in finite dimensional spaces.    
\end{abstract}
\section{Introduction}\label{Section1}
%
In 2020, Liu and alt. \cite{Liu-Livrea-Motreanu-Zeng} introduced a new class of partial differential inclusions, where the driving operator is competing, i.e., it exhibits the \textit{difference} between two elliptic terms, as a $p$-Laplacian and a $q$-Laplacian. This destroys any ellipticity or monotonicity structure, so traditional tools (like nonlinear regularity theory, comparison principles, etc.) don't work. The reaction is multi-valued but not convective, while solutions (in a suitable sense) are obtained via Galerkin's method and Ekeland's variational principle. The same year, D. Motreanu, in \cite{Mo}, studied the case when (single-valued) right-hand sides depend on the gradient of solutions, which clearly forbids variational methods. Since then, some further papers on related topics have been published. Let us mention \cite{Gambera-Marano-Motreanu-JMAA} for convective systems, \cite{Gambera-Marano-Motreanu-FCAA} treating the full fractional case, \cite{Motreanu-CNSNS2023,Mo1} devoted to hemi-variational inequalites, and \cite{Galewski-Motreanu-AML-2024} for operators with gradient depending weight.

As far as we know, except a few papers (see, e.g., \cite{Mo2,RF} and the references therein), the nowadays literature is mainly focused on Dirichlet boundary-value problems without degenerate effects. However, when studying mechanical contact problems between different materials, constitutive laws are usually formulated through differential inclusions, rather than equations, and elastic operators might be described by double- or multi-phase maps; cf. \cite{Zhikov-1986}. Therefore, the question whether the previous technical approaches are adaptable to solve inclusion systems with double-phase, competing differential operators, convection, and mixed boundary conditions naturally arises. 

The main purpose of this work is to provide a positive answer. In fact, let $\Omega$ be a bounded domain in $\real^N$, $N\geq 2$, having a Lipschitz boundary $\partial\Omega$, which splits into two non-trivial parts $\partial\Omega_1$ and $\partial\Omega_2$, with $\partial\Omega_2$ of positive $(N-1)$-dimensional Hausdorff measure.  We will always assume that
\begin{itemize}
\item[$H(0)$] $p_i,q_i\in (1,N)$ and $\mu_i\in L^\infty(\Omega)_+$, with $i=1,2,3,4$,  satisfy
$$p_{2j-1}>p_{2j},\quad q_{2j-1}>q_{2j},\quad\mu_{2j-1}(z)\ge\mu_{2j}(z) \;\;\mbox{a.e. in }\Omega,$$
where $j=1,2$.
\end{itemize}
Consider the modular function $\mathcal G_i:\Omega\times[0,+\infty)\to [0,+\infty)$ given by
\begin{equation}\label{defG}
\mathcal G_i(z,s):=|s|^{p_i}+\mu_i(z)|s|^{q_i}\quad\forall\, (z,s)\in\Omega\times\real, 
\end{equation}
and define
$$\mathscr D_{p_i,q_i,\mu_i}w:=\mbox{div}\left(|\nabla w|^{p_i-2}\nabla w+\mu_i(z)
|\nabla w(z)|^{q_i-2}\nabla w\right),\quad w\in W^{1,\mathcal G_i}(\Omega);$$
see Section \ref{Sections2} for details. The following inclusion systems, that involve double-phase, possibly competing differential operators, convection, and mixed boundary conditions:
\begin{equation}\label{eqnss1}
\left\{\begin{array}{lll}
-\mathscr D_{p_1,q_1,\mu_1}u+\alpha\mathscr D_{p_2,q_2,\mu_2}u\in h_1(z,u,v,\nabla u,\nabla v) & \mbox{ in $\Omega$},\\
w=0 & \mbox{ on $\partial\Omega_1$},\\
-\frac{\partial u}{\partial n_{1,2}}\in g_1(z,u,v) & \mbox{ on $\partial \Omega_2$},
\end{array}\right.
\end{equation}
\begin{equation}\label{eqnss2}
\left\{\begin{array}{lll}
-\mathscr D_{p_3,q_3,\mu_3}v+\beta\mathscr D_{p_4,q_4,\mu_4}v\in h_2(z,u,v,\nabla u,\nabla v) & \mbox{ in $\Omega$},\\
v=0 & \mbox{ on $\partial\Omega_1$},\\
-\frac{\partial v}{\partial n_{3,4}}\in g_2(z,u,v) & \mbox{ on $\partial \Omega_2$},
\end{array}\right.
\end{equation}
will be investigated. Here, $\alpha,\beta\in\real$, $h_j:\Omega\times\real^2\times\real^{2N}\to 2^\real$ and $g_j:\Omega\times\real^2\to 2^\real$, $j=1,2$, are convex compact-valued multifunctions, while
\begin{align*}
\frac{\partial u}{\partial n_{1,2}}:= 
\left[|\nabla u|^{p_1-2}\nabla u+\mu_1|\nabla u|^{q_1-2}\nabla u
-\alpha(|\nabla u|^{p_2-2}\nabla u+\mu_2|\nabla u|^{q_2-2}\nabla u)\right] \cdot \nu,\\
\frac{\partial v}{\partial n_{3,4}}:=
\left[|\nabla v|^{p_3-2}\nabla v+\mu_3|\nabla v|^{q_3-2}\nabla v
-\beta(|\nabla v|^{p_4-2}\nabla v+\mu_4|\nabla v|^{q_4-2}\nabla v)\right] \cdot\nu,
\end{align*} 
with $\nu$ being the outward unit normal vector to $\partial\Omega$. Under suitable assumptions (cf. $H(h)$ and $H(g)$ below) we will prove that \eqref{eqnss1}--\eqref{eqnss2} posses a \textit{generalized or strong generalized solution} $(u,v)\in W^{1,\mathcal G_1}(\Omega)\times W^{1,\mathcal G_2}(\Omega)$, which reduces to a weak one once competing effects disappear, namely $\max\{\alpha,\beta\}<0$.
\section{Preliminaries and functional framework}\label{Sections2}
Let $X,Y$ be two nonempty sets. A multifunction $\Phi:X\to 2^Y$ is a map from $X$ into the family of all nonempty subsets of $Y$. A function $\varphi:X\to Y$ is called a selection of $\Phi$ when $\varphi(x)\in\Phi(x)$ for every $x\in X$. Given $B\subseteq Y$, put 
$$\Phi^-(B):=\{ x\,\mid\, x\in X\mbox{ and }\Phi(x)\cap B\neq\emptyset\}.$$ 
If $X,Y$ are topological spaces and $\Phi^-(B)$ turns out closed in $X$ for all closed set $B\subseteq Y$ then we say that $\Phi$ is upper semi-continuous. Suppose $(X,\mathcal{F})$ is a measurable space and $Y$ is a topological space. The multifunction $\Phi$ is called measurable when $\Phi^-(B)\in\mathcal{F}$ for every open set $B\subseteq Y$.

The result below will be repeatedly useful. It is stated in \cite[p. 215]{ADZ}.
\begin{Proposition}\label{supresult}
Let $F:\Omega\times\real\to 2^\real$ be a closed-valued multifunction such that
\begin{itemize}
\item $z\mapsto F(z,s)$ is measurable for all $s\in\real$,
\item $s\mapsto F(z,s)$ is upper semi-continuous for a.e. $z\in\Omega$,
\end{itemize}
and let $u:\Omega\to\real$ be measurable. Then the multifunction $z\mapsto F(z,u(z))$ admits a measurable selection.
\end{Proposition}
Now, let $(X,\Vert\cdot\Vert)$ be a real normed space with topological dual $X^*$ and duality brackets $\langle\cdot,\cdot\rangle$. Given a nonempty set $A\subseteq X$, define $|A|:=\sup_{x\in A}\Vert x\Vert$. We say that $\varphi:X\to X^*$ enjoys the $(S_+)$-property when
\begin{equation*}
x_n\rightharpoonup x\;\;\mbox{in $X$,}\;\;\limsup_{n\to+\infty}\langle A(x_n),x_n-x\rangle \le 0\implies x_n\to x\;\;\mbox{in $X$.}   
\end{equation*}
A multifunction $\Phi:X\to 2^{X^*}$ is called coercive provided
$$\lim_{\Vert x\Vert\to\infty}\frac{\inf\{\langle x^*,x\rangle\,\mid\, x\in X,\, x^*\in\Phi(x)\}}{\Vert x\Vert}=+\infty\, .$$
The following result is a direct consequence of~\cite[Proposition 3.2.33]{GP1}. 
\begin{Theorem}\label{surjectivetheorem}
Let $X$ be a finite-dimensional normed space and let $\Phi:X\to 2^{X^*}$ be a convex compact-valued multifunction. Suppose $\Phi$ is upper semi-continuous and coercive. Then there exists $\hat{x}\in X$ satisfying $\Phi(\Hat{x})\ni 0$. 
\end{Theorem}
Henceforth, $p'$ indicates the conjugate exponent of $p\geq 1$ while
\begin{equation*}
\Vert u\Vert_p:=\left\{ 
\begin{array}{ll}
\left(\int_{\Omega }|v(z)|^q dz\right)^{1/p} & \text{ if }1\leq p<+\infty, \\ 
\phantom{} &  \\ 
\underset{z\in\Omega}{\esssup}\, |u(z)| & \text{ when } p=+\infty,
\end{array}
\right.
\end{equation*}
Moreover, if $p<N$ then
$$p^*:=\frac{Np}{N-p}\quad\mbox{and}\quad p_*:=\frac{(N-1)p}{N-p}.$$ 
Let $\mathcal{G}_i$ be defined by \eqref{defG}. On the Musielak-Orlicz Lebesgue space
$$L^{\mathcal{G}_i}(\Omega):=\left \{w\,\mid\, w:\Omega\to\real\text{ is measurable and }\int_\Omega\mathcal{G}_i(z,w(z))\, dz< +\infty \right\}$$
we will consider the Luxemburg norm
$$\|w\|_{\mathcal{G}_i}:=\inf\left\{\zeta >0\,\mid\,\int_\Omega \mathcal{G}_i\left(z,\frac{w(z)}{\zeta}\right)dz\leq 1\right\},\quad
w\in L^{\mathcal{G}_i}(\Omega).$$
The associated Musielak-Orlicz Sobolev space is 
$$W^{1,\mathcal{G}_i}(\Omega):=\left\{w\in L^{\mathcal{G}_i}(\Omega) \,\mid\,|\nabla w|\in L^{\mathcal{G}_i}(\Omega) \right\},$$
equipped with the norm
$$\|w\|_{1,\mathcal{G}_i}:=\|\nabla w\|_{\mathcal{G}_i} +\|w\|_{\mathcal{G}_i},\quad w\in W^{1,{\mathcal G}_i}(\Omega).$$
It is known \cite{Mu} that both $L^{\mathcal{G}_i}(\Omega)$ and $W^{1,\mathcal{G}_i}(\Omega)$ turn out complete, separable, and reflexive. Put
\begin{equation*}
U:=\{u\in W^{1,\mathcal{G}_1}(\Omega)\,\mid\, u=0\mbox{ on $\partial \Omega_1$}\},\quad
V:=\{v\in W^{1,\mathcal{G}_3}(\Omega)\,\mid\, v=0\mbox{ on $\partial \Omega_1$}\},
\end{equation*}
as well as
$$\|u\|_U:=\|\nabla u\|_{\mathcal G_1}\quad\forall\, u\in U,\quad
\|v\|_V:=\|\nabla v\|_{\mathcal G_3}\quad\forall\, v\in V.$$
By \cite[Lemma 4]{Zeng-Vetro-Nguyen}, both $(U,\|\cdot\|_U)$ and $(V,\|\cdot\|_V)$ are separable reflexive Banach spaces. The symbols $\langle\cdot,\cdot\rangle_1$, $\langle\cdot,\cdot\rangle_2$, and $\langle\cdot,\cdot\rangle_{3}$ indicate the duality brackets of $U$, $V$, and $U\times V$, respectively, while $\lambda_j$, $j=1,2,3,4$, denote the smallest positive constants such that
\begin{equation}\label{embedone}
\begin{split}
\|u\|_{p_1}\le\lambda_1\|u\|_U\;\;\mbox{and}\;\;
\|u\|_{L^{p_1}(\partial\Omega_2)}\le\lambda_2\|u\|_U\;\;\forall\, u\in U,\\
\|v\|_{p_3}\le\lambda_3\|v\|_V\;\;\mbox{and}\;\;
\|v\|_{L^{p_4}(\partial\Omega_2)}\le\lambda_4\|v\|_V\;\;\forall\, v\in V.
\end{split}
\end{equation}
We associate with the driving operator in \eqref{eqnss1} the map $\mathcal D_1:U\to U^*$ 
given by
\begin{equation}\label{defDone}
\begin{split}
\langle\mathcal D_1 u,\varphi\rangle_1:= & \int_\Omega (|\nabla u|^{p_1-2}\nabla u+\mu_1|\nabla u|^{q_1-2}\nabla u)\cdot\nabla\varphi\,dz\\
& -\alpha\int_\Omega(|\nabla u|^{p_2-2}\nabla u+\mu_2|\nabla u|^{q_2-2}\nabla u)
\cdot\nabla\varphi\,dz,\quad u,\varphi\in U.
\end{split}
\end{equation}
Likewise, for \eqref{eqnss2} one considers the function $\mathcal D_2:V\to V^*$ defined by
\begin{equation}\label{defDtwo}
\begin{split}
\langle\mathcal D_2 v,\psi\rangle_2:= & \int_\Omega (|\nabla v(z)|^{p_3-2}\nabla v+\mu_3|\nabla v|^{q_3-2}\nabla v)\cdot\nabla\psi\,dz\\
& -\beta\int_\Omega(|\nabla v|^{p_4-2}\nabla v+\mu_4|\nabla v|^{q_4-2}\nabla v)
\cdot\nabla\psi\,dz, \quad v,\psi\in V.
\end{split}
\end{equation}

Finally, to shorten notation, set, for every $(u,v)\in U\times V$,
\begin{equation}\label{defD}
\mathcal{D}(u,v):=(\mathcal{D}_1 u,\mathcal{D}_2 v),    
\end{equation}
$$\mathcal{S}_{h_1,h_2}(u,v)=\{(\eta^1,\eta^2)\in 
L^{(p_1^*)'}(\Omega)\times L^{(p_3^*)'}(\Omega)\,\mid\,\eta^i\in h_i(\cdot,u,v,
\nabla u,\nabla v),\, i=1,2 \},$$
$$\mathcal{S}_{g_1,g_2}(u,v)=\{(\xi^1,\xi^2)\in 
L^{(p_{1,*})'}(\partial\Omega_2)\times L^{(p_{3,*})'}(\partial\Omega_2)\,\mid\,\xi^i\in g_i(\cdot,u,v),\, i=1,2 \},$$
where, as usual, $p_j^*:=\frac{Np_j}{N-p_j}$ and $p_{j,*}:=\frac{(N-1)p_j}{N-p_j}$, $j=1,3$.

The following notions of generalized solution and strong generalized solution to problem \eqref{eqnss1}--\eqref{eqnss2} are adapted from those in \cite{Gambera-Marano-Motreanu-FCAA,Gambera-Marano-Motreanu-JMAA}.
\begin{Definition}\label{Definitions1}
We say that $(u,v)\in U\times V$ is a generalized solution of \eqref{eqnss1}--\eqref{eqnss2} if there exist three sequences 
$$(u_n,v_n)\in U\times V,\quad (\eta_n^1,\eta_n^2)\in\mathcal{S}_{h_1,h_2}(u_n,v_n),\quad  (\xi_n^1,\xi_n^2)\in\mathcal{S}_{g_1,g_2}(u_n,v_n)$$
satisfying 
\begin{itemize}
\item[{\rm(i)}] $(u_n,v_n)\wto (u,v)$ in $U\times V$,
\item[{\rm(ii)}] $\mathcal D (u_n,v_n)+(\xi_n^1,\xi_n^2)-(\eta_n^1,\eta_n^2) \wto 0$ in $U^*\times V^*$, and
\item[{\rm(iii)}] $\displaystyle{\lim_{n\to\infty}}\langle\mathcal D (u_n,v_n)+(\xi_n^1,\xi_n^2)-(\eta_n^1,\eta_n^2),(u_n-u,v_n-v)\rangle_{3}=0$.
\end{itemize}
\end{Definition}
\begin{Definition}\label{Definitions2}
$(u,v)\in U\times V$ is called a strongly generalized solution to \eqref{eqnss1}--\eqref{eqnss2} when there are sequences 
$$(u_n,v_n)\in U\times V,\quad (\eta_n^1,\eta_n^2)\in\mathcal{S}_{h_1,h_2}(u_n,v_n),\quad  (\xi_n^1,\xi_n^2)\in\mathcal{S}_{g_1,g_2}(u_n,v_n)$$
fulfilling {\rm (i)}--{\rm(ii)} of Definition \ref{Definitions1} and, moreover, 
\begin{itemize}
\item[${\rm (iii)}'$] $\displaystyle{\lim_{n\to\infty}}\langle(\mathcal D (u_n,v_n),
(u_n-w,v_n-v)\rangle_{3}=0$.
\end{itemize}
\end{Definition}
\section{Existence results}\label{Sections3}
The hypotheses below on the right-hand sides of \eqref{eqnss1}--\eqref{eqnss2} will be posited. 
\begin{itemize}
\item[$H(h)$] The multifunctions $h_i:\Omega\times\real^2\times\real^{2N}\to 2^\real$, $i=1,2$, take convex compact values. Moreover,
\vskip2pt
\noindent $(1)$ $z\mapsto h_i(z,r_1,r_2,\zeta_1,\zeta_2)$ is measurable for all $(r_1,r_2,\zeta_1,\zeta_2)\in\real^2\times\real^{2N}$ and $(r_1,r_2,\zeta_1,\zeta_2)\mapsto h_i(z,r_1,r_2,\zeta_1,\zeta_2)$ is upper semi-continuous for a.e. $z\in\Omega$. 
\vskip2pt
\noindent $(2)$ There exist $m_1,m_2>0$, $\sigma_j,\theta_j>0$, with $j=1,2,3,4$, $\delta_1\in L^{(p_1^*)'}(\Omega)$, and $\delta_2\in L^{(p_3^*)'}(\Omega)$ such that
$$\frac{\sigma_1}{p_1^*}+\frac{\sigma_2}{p_3^*}\le\frac{1}{(p_1^*)'},\quad
\frac{\sigma_3}{p_1^*}+\frac{\sigma_4}{p_3^*}\le\frac{1}{(p_3^*)'},\quad
\frac{\theta_1}{p_1}+\frac{\theta_2}{p_2}\le\frac{1}{(p_1^*)'},\quad
\frac{\theta_3}{p_1}+\frac{\theta_4}{p_2}\le\frac{1}{(p_3^*)'}$$
as well as
\begin{align*}
|h_1 & (z,r_1,r_2,\zeta_1,\zeta_2)|\\
\le & m_1(|r_1|^{p_1^*-1}+|r_2|^\frac{p_3^*}{(p_1^*)'} +|r_1|^{\sigma_1}|r_2|^{\sigma_2}+|\zeta_1|^\frac{p_1}{(p_1^*)'}
+|\zeta_2|^\frac{p_3}{(p_1^*)'}+|\zeta_1|^{\theta_1}|\zeta_2|^{\theta_2})+\delta_1(z),
\end{align*}
\begin{align*}
|h_2 & (z,r_1,r_2,\zeta_1,\zeta_2)|\\
\le & m_2(|r_1|^\frac{p_1^*}{(p_3^*)'}+|r_2|^{p_3^*-1} +|r_1|^{\sigma_3}|r_2|^{\sigma_4}+|\zeta_1|^\frac{p_1}{(p_3^*)'}
+|\zeta_2|^\frac{p_3}{(p_3^*)'}+|\zeta_1|^{\theta_3}|\zeta_2|^{\theta_4})+\delta_2(z)
\end{align*}
for all $(z,r_1,r_2,\zeta_1,\zeta_2)\in\Omega\times\real^2\times\real^{2N}$.
\vskip2pt
\noindent $(3)$ There are $m_3,m_4,m_5,m_6>0$ and $\delta_3,\delta_4\in L^1(\Omega)_+$ satisfying
\begin{align*}
&\eta_1 r_1\le m_3(|r_1|^{p_1}+|r_2|^{p_3})+m_4(|\zeta_1|^{p_1}+|\zeta_2|^{p_3})+\delta_3(z),\\
&\eta_2 r_2\le m_5(|r_1|^{p_1}+|r_2|^{p_3})+m_6(|\zeta_1|^{p_1}+|\zeta_2|^{p_3})+\delta_4(z),
\end{align*} 
for every $\eta_i\in h_i(z,r_1,r_2,\zeta_1,\zeta_2)$, $i=1,2$, $(z,r_1,r_2,\zeta_1,\zeta_2)\in\Omega\times\real^2\times\real^{2N}$.
\item[$H(g)$] The multifunctions $g_i:\partial\Omega_2\times\real^2\to 2^\real$, $i=1,2$, take convex compact values. Moreover,
\vskip2pt
\noindent $(1)$ $z\mapsto g_i(z,r_1,r_2)$ is measurable for all $(r_1,r_2)\in\real^2$ and $(r_1,r_2)\mapsto g_i(z,r_1,r_2)$ is upper semi-continuous for a.e. $z\in \partial \Omega_2$.
\vskip2pt
\noindent $(2)$ There exist $m_7,m_8>0$, $\sigma_j>0$, with $j=5,6,7,8$, $\delta_5\in L^{(p_1^*)'}(\partial\Omega_2)$, and $\delta_6\in L^{(p_3^*)'}(\partial\Omega_2)$ such that
$$\frac{\sigma_5}{p_{1,*}}+\frac{\sigma_6}{p_{3,*}}\le\frac{1}{(p_{1,*})'},\quad \frac{\sigma_7}{p_{1,*}}+\frac{\sigma_8}{p_{3,*}}\le\frac{1}{(p_{3,*})'}$$
as well as 
\begin{equation*}
\begin{split}
|g_1(z,r_1,r_2)|\le m_7\left(|r_1|^{p_{1,*}-1}+|r_2|^\frac{p_{3,*}}{(p_{1,*})'} +|r_1|^{\sigma_5}|r_2|^{\sigma_6}\right)+\delta_5(z),\\
|g_2(z,r_1,r_2)|\le m_8\left(|r_1|^\frac{p_{1,*}}{(p_{3,*})'}+ |r_2|^{p_{3,*}-1}+|r_1|^{\sigma_7}|r_2|^{\sigma_8}\right)+\delta_6(z)
\end{split}
\end{equation*}
for every $(z,r_1,r_2)\in\partial\Omega_2\times\real^2$.
\vskip2pt
\noindent $(3)$ There are $m_9,m_{10}>0$ and $\delta_7,\delta_8\in L^1(\partial\Omega_2)_+$  fulfilling
\begin{equation*}
\xi_1 r_1\le m_9(|r_1|^{p_1}+|r_2|^{p_3})+\delta_7(z),\quad
\xi_2 r_2\le m_{10}(|r_1|^{p_1}+|r_2|^{p_3})+\delta_8(z),
\end{equation*} 
for all $\xi_i\in g_i(z,r_1,r_2)$, $i=1,2$, $(z,r_1,r_2)\in\partial\Omega_2 \times\real^2$.
\end{itemize}
With the notation introduced in Sections \ref{Section1}--\ref{Sections2}, one has
\begin{Theorem}\label{theorem1}
Let $H(0)$, $H(h)$, and $H(g)$ be satisfied. If, moreover,
\begin{equation}\label{coerc}
m_4+m_6+(m_3+m_5)\lambda_{2j-1}+(m_9+m_{10})\lambda_{2j}<1,\;\; j=1,2, 
\end{equation}
then, for every $\alpha,\beta\in\real$, problem \eqref{eqnss1}--\eqref{eqnss2} admits a generalized solution.
\end{Theorem}
\begin{proof}
Pick $(u,v)\in U\times V$. We claim that
$$\mathcal{S}_{h_1,h_2}(u,v)\neq \emptyset\quad\mbox{and}\quad\mathcal{S}_{g_1,g_2}(u,v)\neq\emptyset.$$
In fact, by $H(h)(1)$ and Proposition \ref{supresult}, the multifunction $h_1(\cdot,u,v,\nabla u,\nabla v)$ possesses a measurable selection $\eta_1:\Omega\to\real$. Combining $H(h)(2)$ with H\"older's inequality produces
\begin{align}\label{eqnss3.1}
\|\eta_1\|_{(p_1^*)'}^{(p_1^*)'} 
& \le\int_\Omega\Big[ m_1\Big(|u|^{p_1^*-1}+|v|^\frac{p_3^*}{(p_1^*)'} +|u|^{\sigma_1}|v|^{\sigma_2}\nonumber\\
& \hskip3cm +|\nabla u|^\frac{p_1}{(p_1^*)'}+|\nabla v|^\frac{p_3}{(p_1^*)'}
			+|\nabla u|^{\theta_1}|\nabla v|^{\theta_2}\Big)+\delta_1\Big]^{(p_1^*)'}\,dz\\
& \le C\Big(\|\delta_1\|_{(p_1^*)'}^{(p_1^*)'}
+\|u\|_{p_1^*}^{p_1^*}+\|v\|^{p_3^*}_{p_3^*} +\|\nabla u\|_{p_1}^{p_1}
+\|\nabla v\|_{p_3}^{p_3}\nonumber\\
&\hskip1cm +\|u\|_{p_1^*}^{\sigma_1(p_1^*)'}\,
\|v\|_{\left(\frac{p_1^*}{\sigma_1 (p_1^*)'}\right)'\sigma_2(p_1^*)'}^{\sigma_2(p_1^*)'}
+\|\nabla u\|_{p_1}^{\theta_1 (p_1^*)'}\,
\|\nabla v\|_{\left(\frac{p_1}{\theta_1(p_1^*)'}\right)'\theta_2(p_1^*)'} ^{\theta_2(p_1^*)'}\Big),\nonumber
\end{align}
namely $\eta_1\in L^{(p_1^*)'}(\Omega)$. Since a similar reasoning applies to $h_2(\cdot,u,v,\nabla u,\nabla v)$, we achieve $\mathcal{S}_{h_1,h_2}(u,v)\neq\emptyset$. Likewise, because of $H(g)(1)(2)$, one has $\mathcal{S}_{g_1,g_2}(u,v)\neq\emptyset$. Recalling that $h_i$ and $g_i$ take convex closed values, an elementary computation shows that the multifunctions
$$\mathcal{S}_{h_1,h_2}:U\times V\to 2^{L^{(p_1^*)'}(\Omega)\times L^{(p_3^*)'}(\Omega)},\quad
\mathcal{S}_{g_1,g_2}:U\times V\to 2^{L^{(p_{1,*})'}(\partial\Omega_2) \times L^{(p_{3,*})'}(\partial\Omega_2)}$$
turn out convex closed-valued. Moreover, due to \eqref{eqnss3.1} and its analogues for $h_2$, $g_1$, and $g_2$,
\begin{equation}\label{prop1}
\mbox{$\mathcal{S}_{h_1,h_2}$ and $\mathcal{S}_{g_1,g_2}$ map bounded sets into bounded sets.}    
\end{equation}
We next claim that 
\begin{equation}\label{prop2}
\mbox{$\mathcal{S}_{h_1,h_2}$ and $\mathcal{S}_{g_1,g_2}$ are strongly-weakly upper semi-continuous.}    
\end{equation}
To show this for $\mathcal{S}_{h_1,h_2}$, fix a nonempty weakly closed set $B\subseteq L^{(p_1^*)'}(\Omega)\times L^{(p_3^*)'}(\Omega)$ and choose a sequence $\{(u_n,v_n)\} \subseteq\mathcal{S}_{h_1,h_2}^-(B)$ fulfilling $(u_n,v_n)\to (u,v)$ in $U\times V$. Thus, $\{(u_n,v_n)\}\subseteq U\times V$ turns out bounded. Because of \eqref{prop1}, the same holds as regards
$$\bigcup_{n\in\mathbb{N}}\mathcal{S}_{h_1,h_2}(u_n,v_n)\subseteq L^{(p_1^*)'}(\Omega)\times L^{(p_3^*)'}(\Omega).$$
So, up to sub-sequences, there exists $(\eta_n^1,\eta_n^2)\in\mathcal{S}_{h_1,h_2}(u_n,v_n)\cap B$, $n\in\mathbb{N}$, such that
$$(\eta_n^1,\eta_n^2)\wto(\eta^1,\eta^2)\quad\mbox{in}\quad
L^{(p_1^*)'}(\Omega)\times L^{(p_3^*)'}(\Omega).$$
One evidently has $(\eta^1,\eta^2)\in B$. Mazur's principle provides a sequence $\{(\kappa_n^1,\kappa_n^2)\}$ of convex combinations of $\{(\eta_n^1,\eta_n^2)\}$ satisfying
$$(\kappa_n^1,\kappa_n^2)\to (\eta^1,\eta^2)\;\;\mbox{in}\;\;
L^{(p_1^*)'}(\Omega)\times L^{(p_3^*)'}(\Omega).$$
Since $h_i$ is upper semi-continuous, this entails
$$\eta^i(z)\in h_i(z,u(z),v(z),\nabla u(z),\nabla v(z))\;\;\mbox{for a.e.}\;\; z\in \Omega.$$
Consequently, $(\eta^1,\eta^2)\in\mathcal{S}_{h_1,h_2}(u,v)\cap B$, i.e., $(u,v)\in\mathcal{S}_{h_1,h_2}^-(B)$, as desired. Concerning $\mathcal{S}_{g_1,g_2}$, the argument is analogous.

Now, let $\mathcal{D}$ be given by \eqref{defD}. Through \eqref{defDone}--\eqref{defDtwo} we easily see that
\begin{equation}\label{propD}
\mbox{$\mathcal{D}:U\times V\to U^*\times V^*$ turns out bounded and continuous.}    
\end{equation}
The space $U\times V$ is separable, therefore it possesses a Galerkin's basis, namely a sequence $\{E_n\}$ of linear sub-spaces of $U\times V$ fulfilling:
\begin{itemize}
\item[$({\rm i}_1)$] ${\rm dim}(E_n)<\infty\;\;\forall\, n\in\mathbb{N}$;
\item[$({\rm i}_2)$] $E_n\subseteq E_{n+1}\;\;\forall\, n\in\mathbb{N}$;
\item[$({\rm i}_3)$] $\overline{\cup_{n=1}^{\infty}E_n}=U\times V$.
\end{itemize}
Pick any $n\in\mathbb{N}$. Consider the problem: Find $(u,v)\in E_n$ such that
\begin{equation}\label{eqnss3.2}
\mathcal D(u,v)+\mathcal{S}_{g_1,g_2}(u,v)-\mathcal{S}_{h_1,h_2}(u,v)\ni 0\;\;\mbox{in}\;\; E_n^*.
\end{equation}
By \eqref{prop1}--\eqref{propD} the multifunction
$$(\mathcal{D}+\mathcal{S}_{g_1,g_2}-\mathcal{S}_{h_1,h_2})\lfloor_{E_n}:E_n\to 2^{E_n^*}$$
takes convex closed values, maps bounded sets into bounded sets, and is upper semi-continuous. If $(u,v)\in U\times V$, $(\xi_1,\xi_2)\in\mathcal{S}_{g_1,g_2}(u,v)$, $(\eta_1,\eta_2)\in\mathcal{G}_{h_1,h_2}(u,v)$ then, thanks to $H(h)(3)$ and $H(g)(3)$, 
\begin{align*}
\langle\mathcal D(u,v) & +(\xi_1,\xi_2)-(\eta_1,\eta_2),(u,v)\rangle_{3}\geq
\|\nabla u\|_{p_1}^{p_1}+\|\nabla v\|_{p_3}^{p_3}
-|\alpha|\|\nabla u\|_{p_2}^{p_2}-|\beta|\|\nabla v\|_{p_4}^{p_4}\\
& + \int_\Omega\mu_1|\nabla u|^{q_1} dz+\int_\Omega\mu_3|\nabla v|^{q_3} dz
-|\alpha|\int_\Omega\mu_2|\nabla u|^{q_2} dz
-|\beta|\int_\Omega\mu_4|\nabla v|^{q_4} dz\nonumber\\
& - \int_\Omega\left[ m_3(|u|^{p_1}+|v|^{p_3})+m_4(|\nabla u|^{p_1}+|\nabla v|^{p_3}) +\delta_3\right] dz\nonumber\\
& -\int_\Omega\left[ m_5(|u|^{p_1}+|v|^{p_3})+m_6(|\nabla u|^{p_1}+|\nabla v|^{p_3}) +\delta_4\right] dz\nonumber\\
&-\int_{\partial\Omega_2}\left[ m_9(|u|^{p_1}+|v|^{p_3})+\delta_7\right] d\sigma
-\int_{\partial\Omega_2}\left[m_{10}(|u|^{p_1}+|v|^{p_3})+\delta_8\right] d\sigma.
\end{align*}
Recalling that $p_{2j-1}>p_{2j}$, $j=1,2$, the Young inequality with $\varepsilon$ produces
$$\|\nabla u\|_{p_2}^{p_2}\leq c(\varepsilon)\|\nabla u\|_{p_1}^{p_1}+C(\varepsilon), \quad\|\nabla v\|_{p_4}^{p_4}\leq c(\varepsilon)\|\nabla v\|_{p_3}^{p_3}+C(\varepsilon),$$
where $c(\varepsilon)\to 0^+$ as $\varepsilon\to 0^+$. Likewise, 
\begin{equation*}
\begin{split}
\int_\Omega\mu_2|\nabla u|^{q_2} dz\leq 
c(\varepsilon)\int_\Omega\mu_1|\nabla u|^{q_1} dz+C(\varepsilon),\\
\int_\Omega\mu_4|\nabla v|^{q_4} dz\leq 
c(\varepsilon)\int_\Omega\mu_3|\nabla v|^{q_3} dz+C(\varepsilon),
\end{split}    
\end{equation*}
because $q_{2j-1}>q_{2j}$ and $\mu_{2j-1}\geq \mu_{2j}$, $j=1,2$, by hypothesis $H(0)$. Via \eqref{embedone} we thus achieve
\begin{align*}
\langle\mathcal D(u,v) & +(\xi_1,\xi_2)-(\eta_1,\eta_2),(u,v)\rangle_{3}\nonumber\\
\ge [& 1-m_4-m_6-(m_3+m_5)\lambda_1-(m_9+m_{10})\lambda_2-\hat{c}(\varepsilon)]
\|\nabla u\|_{p_1}^{p_1}\nonumber\\
& +[1-m_4-m_6-(m_3+m_5)\lambda_3-(m_9+m_{10})\lambda_4-\Tilde{c}(\varepsilon)]
\|\nabla v\|_{p_3}^{p_3}
\nonumber\\
& +[1-\hat{c}(\varepsilon)]\int_\Omega\mu_1|\nabla u|^{q_1} dz
+[1-\Tilde{c}(\varepsilon)]\int_\Omega\mu_3|\nabla v|^{q_3} dz+C_1(\varepsilon)\nonumber\\
& -\|\delta_3\|_1-\|\delta_4\|_1-\|\delta_7\|_{L^1(\partial\Omega_2)}
-\|\delta_8\|_{L^1(\partial\Omega_2)},
\end{align*}
with $\hat{c}(\varepsilon):=|\alpha| c(\varepsilon)$, $\Tilde{c}(\varepsilon):=|\beta| c(\varepsilon)$, and appropriate $C_1(\varepsilon)>0$. Put
\begin{equation*}
\begin{split}
A(\varepsilon):=1-m_4-m_6-(m_3+m_5)\lambda_1-(m_9+m_{10})\lambda_2-\hat{c}(\varepsilon),\\
B(\varepsilon):=1-m_4-m_6-(m_3+m_5)\lambda_3-(m_9+m_{10})\lambda_4-\Tilde{c}(\varepsilon).\\
C_2(\varepsilon):=C_1(\varepsilon)-\|\delta_3\|_1-\|\delta_4\|_1
-\|\delta_7\|_{L^1(\partial\Omega_2)}-\|\delta_8\|_{L^1(\partial\Omega_2)}.
\end{split}   
\end{equation*}
In view of \eqref{coerc} one has $A(\varepsilon),B(\varepsilon)>0$ provided $\varepsilon>0$ is small enough. Therefore, the previous inequality clearly entails
\begin{equation}\label{eqnss3.3}
\begin{split}
\langle\mathcal D(u,v) & +(\xi_1,\xi_2)-(\eta_1,\eta_2),(u,v)\rangle_{3}\\
\ge & A(\varepsilon)\min\{\| u\|_U^{p_1},\| u\|_U^{q_1}\}
+B(\varepsilon)\min\{\| v\|_{V}^{p_3},\| v\|_{V}^{q_3}\}+C_2(\varepsilon),
\end{split}
\end{equation}
i.e., the multifunction $(\mathcal{D}+\mathcal{S}_{g_1,g_2}-\mathcal{S}_{h_1,h_2})\lfloor_{E_n}$ turns out coercive.
%
%
Now, Theorem~\ref{surjectivetheorem} can be applied, and there exists a solution $(u_n,v_n)\in E_n$ to problem \eqref{eqnss3.2}, whence
\begin{equation}\label{eqnss3.4}
\mathcal D(u_n,v_n)+(\xi_n^1,\xi_n^2)-(\eta_n^1,\eta_n^2)= 0\;\;\mbox{in}\;\; E^*_n 
\end{equation}
for suitable $(\xi_n^1,\xi_n^2)\in \mathcal{S}_{g_1,g_2}(u_n,v_n)$, $(\eta_n^1,\eta_n^2) \in\mathcal{G}_{h_1,h_2}(u_n,v_n)$. From \eqref{eqnss3.3}, written with $(u,v):=(u_n,v_n)$, and \eqref{eqnss3.4} it follows
$$0\geq A(\varepsilon)\min\{\| u_n\|_U^{p_1},\| u_n\|_U^{q_1}\}+
B(\varepsilon)\min\{\| v_n\|_{V}^{p_3},\| v_n\|_{V}^{q_3}\}+C_2(\varepsilon)\quad\forall\, n\in\mathbb{N}.$$
So, $\{(w_n,v_n)\}\subseteq U\times V$ is bounded. By reflexivity one has $(u_n,v_n)\wto (u,v)$ in $U\times V$, taking a sub-sequence when necessary. Consequently, (i) of Definition \ref{Definitions1} holds. Through \eqref{propD} and \eqref{eqnss3.4} we next infer that
$\{(\eta_n^1,\eta_n^2)\}\subseteq L^{(p_1^*)'}(\Omega)\times L^{(p_3^*)'}(\Omega)$
and 
$\{(\xi_n^1,\xi_n^2)\}\subseteq L^{(p_{1,*})'}(\partial\Omega_2)\times L^{(p_{3,*})'}(\partial\Omega_2)$
turn out bounded. Therefore, always up to sub-sequences,
\begin{equation}\label{weaklim}
\mathcal{D} (u_n,v_n)+(\xi_n^1,\xi_n^2)-(\eta_n^1,\eta_n^2)\wto (u^*,v^*)\;\;\mbox{in} \;\; U^*\times V^*.
\end{equation}
Now, given any $(\varphi,\psi)\in \cup_{n=1}^\infty E_n$, Property $({\rm i}_2)$ and \eqref{eqnss3.4} yield
$$\langle (u^*,v^*),(\varphi,\psi)\rangle_3=\lim_{n\to\infty}\langle\mathcal{D}(u_n,v_n)+(\xi_n^1,\xi_n^2)-(\eta_n^1,\eta_n^2),(\varphi,\psi)\rangle_3=0.$$
Because of $({\rm i}_3)$ this forces 
\begin{equation}\label{zerosol}
(u^*,v^*)= 0\;\;\mbox{in}\;\; U^*\times V^*,    
\end{equation}
namely condition (ii) is true. Finally, through \eqref{eqnss3.4}, \eqref{weaklim}, and \eqref{zerosol} we arrive at
\begin{equation*}
\begin{split}
\langle\mathcal{D}(u_n,v_n)&+(\xi_n^1,\xi_n^2)-(\eta_n^1,\eta_n^2),(u_n-u,v_n-v)\rangle_3\\
& =-\langle\mathcal{D}(u_n,v_n)+(\xi_n^1,\xi_n^2)-(\eta_n^1,\eta_n^2), (u,v)\rangle_3\to 0
\end{split}
\end{equation*}
as $n\to\infty$, which shows (iii) in Definition \ref{Definitions1}. Summing up, the pair $(u,v)$ turns out a generalized solution to \eqref{eqnss1}--\eqref{eqnss2}.
\end{proof}
The existence of strongly generalized solutions will be established under the following assumptions.
\begin{itemize}
\item[$H(h)$] $(2')$ There exist $d_1,d_2>0$ $\sigma_j,\theta_j>0$, with $j=1,2,3,4$, $1<\kappa_1<p_1^*$, $1<\kappa_2<p_3^*$, and $\pi_j\in L^{(\kappa_j)'}(\Omega)$, $j=1,2$, such that
$$\frac{\sigma_1}{\kappa_1}+\frac{\sigma_2}{\kappa_2}\le\frac{1}{\kappa_1'},\quad \frac{\sigma_3}{\kappa_1}+\frac{\sigma_4}{\kappa_2}\le\frac{1}{\kappa_2'},\quad \frac{\theta_1}{p_1}+\frac{\theta_2}{p_2}\le \frac{1}{\kappa_1'},\quad
\frac{\theta_3}{p_1}+\frac{\theta_4}{p_2}\le \frac{1}{\kappa_2'}$$
as well as 
\begin{align*}
|h_1 & (z,r_1,r_2,\zeta_1,\zeta_2)|\\
\le & d_1\left(|r_1|^\frac{p_1^*}{\kappa_1'}+|r_2|^\frac{p_3^*}{\kappa_1'} +|r_1|^{\sigma_1}|r_2|^{\sigma_2}+|\zeta_1|^\frac{p_1}{\kappa_1'}
+|\zeta_2|^\frac{p_3}{\kappa_1'}+|\zeta_1|^{\theta_1}|\zeta_2|^{\theta_2}\right) +\pi_1(z),\\
|h_2 & (z,r_1,r_2,\zeta_1,\zeta_2)|\\
\le & d_1\left(|r_1|^\frac{p_1^*}{\kappa_2'}+|r_2|^\frac{p_3^*}{\kappa_2'} +|r_1|^{\sigma_3}|r_2|^{\sigma_4}+|\zeta_1|^\frac{p_1}{\kappa_2'}
+|\zeta_2|^\frac{p_3}{\kappa_2'}+|\zeta_1|^{\theta_3}|\zeta_2|^{\theta_4}\right)
+\pi_2(z)
\end{align*}
for every $(z,r_1,r_2,\zeta_1,\zeta_2)\in\Omega\times\real^2\times\real^{2N}$.
\item[$H(g)$] $(2')$ There are $d_3,d_4>0$, $\sigma_j>0$, with $j=5,6,7,8$, $1<\kappa_3<p_{1,*}$, $1<\kappa_4<p_{3,*}$, and $\pi_j\in L^{\kappa_j'}(\partial\Omega_2)$, $J=3,4$, fulfilling
$$\frac{\sigma_5}{\kappa_3}+\frac{\sigma_6}{\kappa_4}\le\frac{1}{\kappa_3'},\quad
\frac{\sigma_7}{\kappa_3}+\frac{\sigma_8}{\kappa_4}\le\frac{1}{\kappa_4'},$$
as well as 
\begin{align*}
& |g_1(z,r_1,r_2)|\le d_3\left(|r_1|^\frac{p_{1,*}}{\kappa_3'}
+|r_2|^\frac{p_{3,*}}{\kappa_3'}
+|r_1|^{\sigma_5}|r_2|^{\sigma_6}\right)+\pi_3(z),\\
& |g_2(z,r_1,r_2)|\le d_4\left(|r_1|^\frac{p_{1,*}}{\kappa_4'}
+|r_2|^\frac{p_{3,*}}{\kappa_4'}
+|r_1|^{\sigma_7}|r_2|^{\sigma_8}\right)+\pi_4(z)
\end{align*}
for all $(z,r_1,r_2)\in\partial\Omega_2\times\real^{2}$.
\end{itemize} 
\begin{Theorem}
If $H(0)$, $H(h)(1)(2')(3)$, $H(g)(1)(2')(3)$, and \eqref{coerc} are satisfied then, for every $\alpha,\beta\in\real$, problem \eqref{eqnss1}--\eqref{eqnss2} admits a strong generalized solution $(u,v)$. Moreover, $(u,v)$ is a weak solution once $\max\{\alpha,\beta\}<0$.
\end{Theorem}
\begin{proof}
Since $H(h)(2')$ and $H(g)(2')$ are stronger that $H(h)(2)$ and $H(g)(2)$, respectively, we can use Theorem \ref{theorem1}, which provides $(u,v)\in U\times V$ as well as three bounded sequences
\begin{equation*}
\begin{split}
(u_n,v_n) & \in U\times V,\\
(\eta_n^1,\eta_n^2) & \in\mathcal{S}_{h_1,h_2}(u_n,v_n)\cap(L^{(\kappa_1)'}(\Omega)\times L^{(\kappa_2)'}(\Omega)),\\
(\xi_n^1,\xi_n^2) & \in\mathcal{S}_{g_1,g_2}(u_n,v_n)\cap(L^{(\kappa_3)'}(\partial\Omega_2) \times L^{(\kappa_4)'}(\partial\Omega_2))    
\end{split}   
\end{equation*}
fulfilling conditions (i)--(iii) in Definition \ref{Definitions1}. From $\kappa_1<p_1^*$, $\kappa_3<p_{1,\ast}$, $\kappa_2<p_3^*$, and $\kappa_4<p_{3,\ast}$ it follows that both embeddings
$$U\times V\hookrightarrow L^{\kappa_1}(\Omega)\times L^{\kappa_2}(\Omega),\quad
U\times V\hookrightarrow L^{\kappa_3}(\partial\Omega_2)\times L^{\kappa_4}(\partial\Omega_2)$$
turn out compact. Hence, up to sub-sequences,
$$(u_n,v_n)\to (u,v)\;\;\mbox{in}\;\; L^{\kappa_1}(\Omega)\times L^{\kappa_2}(\Omega)\;\;
\mbox{and in}\;\; L^{\kappa_3}(\partial\Omega_2)\times L^{\kappa_4}(\partial\Omega_2).$$
This evidently implies
\begin{equation}\label{splus}
\begin{split}
0= & \lim_{n\to\infty}\langle(\mathcal{D}(u_n,v_n)+(\xi_n^1,\xi_n^2)-(\eta_n^1,\eta_n^2),(u_n-u,v_n-v)\rangle_3\\
= & \lim_{n\to\infty}\langle(\mathcal{D}(u_n,v_n),(u_n-u,v_n-v)\rangle_3,
\end{split}
\end{equation} 
i.e., ${\rm (iii)}'$ of Definition \ref{Definitions2} also holds. Suppose now $\max\{ \alpha,\beta\}<0$. Then the operators $\mathcal{D}_i$, $i=1,2$, enjoy the $(S_+)$-property, as this makes the classical $(p,q)$-Laplacian (cf., e.g., \cite[Proposition 2.1]{GM}). Thus, \eqref{splus}
%
%
forces
\begin{equation}\label{strongconv}
(u_n,v_n)\to (u,v)\;\;\mbox{in}\;\;U\times V.    
\end{equation}
By boundedness we then get
\begin{equation}\label{weakconv}
\begin{split}
(\eta^1_n,\eta^2_n) & \wto(\eta^1,\eta^2)\;\;\mbox{in}\;\; L^{\kappa_1'}(\Omega)\times L^{\kappa_2'}(\Omega),\\
(\xi_n^1,\xi_n^2) & \wto (\xi^1,\xi^2)\;\;\mbox{in}\;\; L^{\kappa_3'}(\partial\Omega_2) \times L^{\kappa_4'}(\partial\Omega_2). 
\end{split}
\end{equation}
So, thanks to \eqref{prop2} and \cite[Proposition 2.5]{Sm}, 
\begin{equation}\label{final}
(\eta^1,\eta^2)\in\mathcal{S}_{h_1,h_2}(u,v),\quad(\xi^1,\xi^2)\in \mathcal{S}_{g_1,g_2}(u,v).
\end{equation}
Finally, because of (ii) in Definition \ref{Definitions1}, \eqref{strongconv}, and \eqref{propD}, for every fixed $(\varphi,\psi)\in U\times V$ one has
\begin{equation*}
\begin{split}
0= & \lim_{n\to\infty}\langle\mathcal{D}(u_n,v_n)+(\xi_n^1,\xi_n^2)-(\eta_n^1,\eta_n^2),(\varphi,\psi)\rangle_3\\
= & \langle(\mathcal{D}(u,v)+(\xi^1,\xi^2)-(\eta^1,\eta^2),(\varphi,\psi)\rangle_3,   
\end{split}
\end{equation*}
whence $\mathcal{D}(u,v)+(\xi^1,\xi^2)-(\eta^1,\eta^2)=0$ in $U^*\times V^*$. Taking \eqref{final} into account, we see that $(u,v)$ is a weak solution to \eqref{eqnss1}--\eqref{eqnss2}.
\end{proof}
\section*{Acknowledgment}
This work was supported in part by the Natural Science Foundation of Guangxi under Grant Nos. 2021GXNSFFA196004 and 2024GXNSFBA010337, the National Natural Science Foundation of China under Grant No. 12371312, and the Natural Science Foundation of Chongqing under Grant No. CSTB2024NSCQ-JQX0033.  

The second author was partly funded by: Research project of MIUR (Italian Ministry of Education, University and Research) Prin 2022 {\it Nonlinear differential problems with applications to real phenomena} (Grant No. 2022ZXZTN2).

The second author is a member of the {\em Gruppo Nazionale per l'Analisi Matematica, la Probabilit\`a e le loro Applicazioni} (GNAMPA) of the {\em Istituto Nazionale di Alta Matematica} (INdAM).
\vskip5pt
\noindent {\bf Authors' contributions}\\
All the authors contributed equally to the writing of this paper, read and approved the final manuscript.

\end{document}